\documentclass[11pt]{amsart}%

\usepackage{graphicx}
\usepackage[margin=1 in, top = 0.8 in]{geometry}

\def\C{\mathbb{C}}
\def\R{\mathbb{R}}

\def\Z{\mathbb{Z}}

\def\P{\mathbb{P}}

\def\cO{\mathcal{O}}

\def\cW{\mathcal{W}}
\def\cM{\mathcal{M}}
\def\cV{\mathcal{V}}

\def\Jac{\operatorname{Jac}}

\usepackage{mathtools}
\usepackage{xcolor}
\usepackage{amssymb}
\usepackage{amsmath}
\usepackage{amssymb,amsmath, calligra,mathrsfs}
\usepackage{tikz}
\usepackage{tikz-cd}
\usepackage{tikz-3dplot}
\usetikzlibrary{shapes,calc,positioning}
\usepackage{tkz-euclide}
\usepackage{pgfplots}
\usepgfplotslibrary{polar}
\usepackage[caption=false]{subfig}
\usepackage{amssymb,amsmath,amsthm, calligra,mathrsfs}
\usepackage[english]{babel}
\usepackage{hyphenat}
\usepackage{float}
\usepackage{comment}

\usepackage[pagewise]{lineno}

\newtheorem{theorem}{Theorem}[section]

\newtheorem{lemma}[theorem]{Lemma}
\newtheorem{proposition}[theorem]{Proposition}

\newtheorem*{theorem*}{Theorem}
\theoremstyle{definition}

\newtheorem{example}[theorem]{Example}

\newenvironment{polynomial}
  {\par\vspace{\abovedisplayskip}%
   \setlength{\leftskip}{\parindent}%
   \setlength{\rightskip}{\leftskip}%
   \medmuskip=4mu plus 2mu minus 2mu
   \binoppenalty=0
   \noindent$\displaystyle}
  {$\par\vspace{\belowdisplayskip}}

\usepackage{ucs}
\usepackage[utf8x]{inputenc}

\begin{document}
\title{Multi-Degrees in Polynomial Optimization}

\author{Kemal Rose\\ \small{ MPI MiS, Inselstrasse 22, 04103, Leipzig, Germany} 
\email{kemal.rose@mis.mpg.de}}
\date{\today}

\date{}

%

%
%


%
\maketitle              

\begin{abstract}
We study structured optimization problems with polynomial objective function and polynomial equality constraints. The structure comes from a multi-grading on the polynomial ring in several variables. For fixed multi-degrees we determine the generic number of complex critical points. This serves as a measure for the algebraic complexity of the optimization problem. We also discuss computation and certification methods coming from numerical nonlinear algebra.
\end{abstract}

\section{Introduction}

Consider the following optimization problem:
\begin{equation}
\begin{gathered}
    \operatorname{minimize} \; \ f_0(x), \ x \in \mathbb{R}^n \quad
    \\
     \text{subject to } f_1(x) = \cdots = f_m(x) = 0
\end{gathered}
\label{eq:rewardMaximization}
\end{equation}
Here $f_0, \dots, f_m$ are polynomials in the ring $\mathbb{R}[X]$ in 
$n$ variables, partitioned into $k$ subsets
\[
X = \bigcup_{i = 1}^k  X_i, \ X_i = ( x_{i, 1}, \dots, x_{i, n_i})
\]
with $n = n_1 + \cdots + n_k$.
To find the global optimizers of $\eqref{eq:rewardMaximization}$ we investigate the critical points of $f_0$ restricted to the feasible region.
Since each global optimizer is amongst these critical points,
it is an important problem in polynomial optimization to understand how they depend on the data
$f_0, \dots, f_m$.
Following \cite{algebraic_degree}
we call the number of complex critical points of \eqref{eq:rewardMaximization} the \emph{algebraic degree of optimization}.
These degrees constitute fundamental invariants of the variety
$\{f_1 = \cdots = f_m = 0\}$, including the \emph{euclidean distance} degree \cite{ED_deg}, the \emph{maximum likelihood degree} \cite{ML_deg},
and are still subjects of active research \cite{Kaie}.

\begin{example}[$n = 2, \ k = 1$]
The optimization problem \eqref{eq:rewardMaximization} with
\[
f_0 = x^2 + y^2, \ f_1 = -x^2 + y^2 + 1
\]
has two real critical points
displayed in figure \ref{fig:level_lines}.

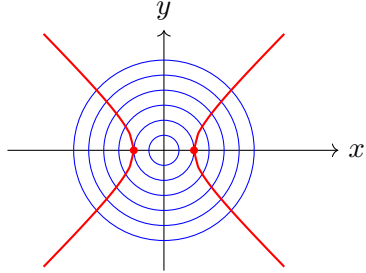
\begin{figure}[H]

\begin{center}

\begin{tikzpicture}[scale = 0.8]
  \draw[->] (-2.6, 0) -- (2.9, 0) node[right] {$x$};
  \draw[->] (0, -2) -- (0, 2) node[above] {$y$};

\draw[blue, thin] (0,0) circle (0.25 cm);
\draw[blue, thin] (0,0) circle (0.5 cm);
\draw[blue, thin] (0,0) circle (0.75 cm);
\draw[blue, thin] (0,0) circle (1 cm);
\draw[blue, thin] (0,0) circle (1.25 cm);
\draw[blue, thin] (0,0) circle (1.5 cm);

  \draw[scale=0.5, domain=1:4, smooth, variable=\x, red, thick]  plot ({\x}, {  sqrt(\x^2 - 1) });
  \draw[scale=0.5, domain=1:4, smooth, variable=\x, red, thick]  plot ({\x}, {  -sqrt(\x^2 - 1) });

  \draw[scale=0.5, domain=1:4, smooth, variable=\x, red, thick]  plot ({-\x}, {  sqrt(\x^2 - 1) });
  \draw[scale=0.5, domain=1:4, smooth, variable=\x, red, thick]  plot ({-\x}, {  -sqrt(\x^2 - 1) });
  
    \node[circle,fill=red,inner sep=0pt,minimum size=3pt,label=below:{}] (a) at (0.5,0) {};
    \node[circle,fill=red,inner sep=0pt,minimum size=3pt,label=below:{}] (a) at (-0.5,0) {};

\end{tikzpicture}
\caption{Level lines}
\label{fig:level_lines}

\end{center}
\end{figure}

\end{example}
In \cite{algebraic_degree} the algebraic degree of optimization has been 
studied not for a specific choice of $f_0$, but instead it is
determined for a generic choice of polynomials of fixed degrees.
To weaken this restrictive genericity assumption it is desirable to replace the degree of a polynomial with a more refined notion. In this spirit we denote for every polynomial $f$ in $\mathbb{R}[X]$ by $\operatorname{deg}(f) \in \mathbb{Z}_{\geq 0}^k$
its multi-degree. This is the vector of its degrees when considered as a polynomial in the $k$ different sets of variables $X_i$.

The paper is structured as follows.
In Section \ref{sec:critical_points} we give homogeneous equations for the critical points.
In Section $3$ we
determine the algebraic degree of optimization for a generic choice of polynomials
$f_0, \dots, f_m$ with fixed multi-degrees, using methods from intersection theory.
Our main result is Theorem \ref{thm: algebraic_degree_of_optimization}.
In the final section, Section  \ref{sec: numerical_nonlinear_algebra}, we discuss how numerical methods
and bounds for the algebraic degree of optimization can be used in conjunction, in order to
obtain optimizers of \eqref{eq:rewardMaximization} together with a certificate for the correctness of the result.
This uses an implementation of homotopy continuation \cite{HC.jl}, together with a certification method that we implemented,  based on interval arithmetic \cite{interval_certification}.

\section{Critical point equations}
\label{sec:critical_points}
From now on the polynomials $f_i$ are generic polynomials with fixed, positive multidegrees $\deg(f_i) \in \Z_{\geq 0}^k$.
A standard way of
solving \eqref{eq:rewardMaximization} is to apply Lagrange multipliers.
In particular, the constrained optimization problem \eqref{eq:rewardMaximization}
is replaced with the unconstrained problem
\begin{equation}
\label{eq:Lagrangian}
\nabla \mathcal{L} = 0,
\end{equation}
where $\nabla \mathcal{L}$ is the gradient of the Lagrangian
\[\mathcal{L} =  f_0 - \lambda_1f_1 -  \dots - \lambda_mf_m.\]

For the purposes of section \ref{sec: counting_pts} we will also consider the following alternative formulation.
A smooth point $x$ on the affine variety
\[
V := \{ x \in \C^n | \  f_1(x) = \cdots = f_m(x) = 0   \}
\]
is a critical point of $f_0$ if the rank of the Jacobian
\[
M := \Jac(f_0,f_1,  \dots, f_m) = (\nabla f_0, \nabla f_1,	\dots, \nabla f_m)
\]
drops at $x$.
Note that, since all multi-degrees are strictly positive,
by Bertini's theorem the variety $V$ is smooth and of codimension $m$, so the Jacobian
$(\nabla f_1, \dots, \nabla f_m)$ has full rank at every point in $V$.
Let now $W$ denote the affine variety determined by the vanishing of the maximal minors of $M$.

With the aim of computing the cardinality of the intersection $V \cap W$ using methods from intersection theory, we consider the closure $\cV$ of $V$ and $\cW$ of $W$ in the product of projective spaces
\[
\mathcal{ X}  = \P^{n_1}  \times \cdots \times \P^{n_k}.
\]
We start by giving homogeneous defining equations:
denote for every polynomial $f$ in $\R[x]$ by $\tilde{f}$ its multihomogenization 
\[
\tilde{f} = x_{1, 0}^{d_1} \cdots x_{k, 0}^{d_k} f \left( \frac{X_1}{x_{1, 0}}, \dots, \frac{X_k}{x_{k, 0}} \right),
\]
where the multi-degree of $f$ is $(d_1, \dots, d_k)$.
The vanishing locus $\cV$ of $\tilde{f_1}, \dots, \tilde{f_m}$ is the Zariski-closure of $V$ in $\mathcal{X}$.
Analogously, let
$\cM$ denote the $n \times m+1$ matrix
\begin{equation*}
\cM = 
\begin{pmatrix}
\frac{\partial}{\partial{x_{1,1}}} \tilde{f}_0 &  \cdots & \frac{\partial}{\partial{x_{1,1}}} \tilde{f}_m \\
\vdots    & \ddots & \vdots  \\
\frac{\partial}{\partial{x_{1, n_1}}} \tilde{f}_0  & \cdots &\frac{\partial}{\partial{x_{1, n_1}}} \tilde{f}_m\\
\vdots    & \vdots & \vdots  \\
\frac{\partial}{\partial{x_{k,1}}} \tilde{f}_0 &  \cdots & \frac{\partial}{\partial{x_{k,1}}} \tilde{f}_m \\
\vdots    & \ddots & \vdots  \\
\frac{\partial}{\partial{x_{k, n_k}}} \tilde{f}_0  & \cdots &\frac{\partial}{\partial{x_{k, n_k}}} \tilde{f}_m
\end{pmatrix}.
\end{equation*}
The $m+1$ by $m+1$ minors of $\cM$ 
are the homogenizations of the minors of $M$.
Their vanishing locus is 
the closure $\cW$ of $W$ in $\mathcal{X}$.

\section{Counting critical points}
\label{sec: counting_pts}
Our first observation is the following:
\begin{proposition}
The intersection $\cV \cap \cW$ is finite.
\end{proposition}
\begin{proof}
Assume that there is a curve $C$ contained in 
$\cV \cap \cW$. Then $C$ intersects the vanishing locus of $\tilde{f_0}$. Equivalently there is a point $x$ contained in
$\cV \cap \{\tilde{f}_0 = 0\}$ where the maximal minors of $\cM$ vanish.
Since by Bertini's Theorem the intersection
$\cV \cap \{ \tilde{f}_0 = 0\}$ is a smooth variety,
the maximal minors of the matrix 
\[
\Jac(\tilde{f}_0, \dots, \tilde{f}_m)
 = 
\begin{pmatrix}
\frac{\partial}{\partial{x_{1,0}}} \tilde{f}_0 &  \cdots & \frac{\partial}{\partial{x_{1,0}}} \tilde{f}_m \\
\vdots    & \ddots & \vdots  \\
\frac{\partial}{\partial{x_{k, n_k}}} \tilde{f}_0  & \cdots &\frac{\partial}{\partial{x_{k, n_k}}} \tilde{f}_m
\end{pmatrix}
\]
do not vanish simultaneously.
This matrix differs from the matrix $\cW$ by some 
additional rows corresponding to the differentials
$\frac{\partial}{\partial{x_{1, 0}}}, \dots, \frac{\partial}{\partial{x_{k, 0}}}$.
However, as a consequence of the Euler relation, the rank of $\cM$ and $\Jac(\tilde{f}_0, \dots, \tilde{f}_m)$
coincide on $\cV$, away from $\cup_{i =1}^k\{x_{i,0} = 0\}$.
This gives us the desired contradiction.
\end{proof}

Applying a Bertini-type argument to \eqref{eq:Lagrangian} shows that the intersection $\cV \cap \cW$
is transversal, implying the following lemma:
\begin{lemma}
$f_0$ has finitely many critical points on $V$. Their number is
identified with the intersection product $[\cV] [\cW]$ in the cohomology ring
$A^*(  \mathcal{X}   ) = \Z[ Y_1, \dots, Y_k]/( Y_1^{n_1+1}, \dots, Y_k^{n_k+1}  )$.
\end{lemma}

As $\cV$ is a transversal intersection of generic global sections of the line bundles
$ \cO(\deg(f_j))$ on $\mathcal{X}$, its class in 
$A^*(\mathcal{X})$ is the product of Chern classes
\[
[\cV] = \prod_{j = 1}^m c_1( \cO(\deg(f_j))  ).
\]

We now compute the class 
$[ \cW ]$.
The matrix $\cM$ may be considered as a collection of linear maps parametrized by $\mathcal{X}$.
In fact, each block
$$
\begin{pmatrix}
\frac{\partial}{\partial{x_{i,1}}} \tilde{f}_0 &  \cdots & \frac{\partial}{\partial{x_{i,1}}} \tilde{f}_m \\
\vdots    & \ddots & \vdots  \\
\frac{\partial}{\partial{x_{i, n}}} \tilde{f}_0  & \cdots &\frac{\partial}{\partial{x_{i, n_1}}} \tilde{f}_m\\
\end{pmatrix}
$$
of $\cM$
has elements of multi-degree $\deg(f_j) - e_i$ in the $j$-th cloumn.
$e_i$ denotes the $i$-th vector of unity.
In particular it
determines a map of vector bundles
\[
\cO(e_i)  ^{\oplus n_i} \longrightarrow \bigoplus_{j = 0}^m \cO(\deg(f_j))
\]
on $\mathcal{X}$.
Here $\cO(e_i)$ is the $i-$th canonical line bundle on the product of projective spaces $\mathcal{X}$.
We transpose
$\cM$ to get a map of bundles
\[
\bigoplus_{j = 0}^m \cO(-\deg(f_j))
\longrightarrow 
\bigoplus_{i = 1}^k \cO(-e_i) ^{\oplus n_i}.
\]
On a smooth variety, and under the condition that the codimension of the degeneracy locus $\cW$ of a morphism $\cM$ of vector bundles is appropriate,
the Thom-Porteous-Giambelli formula \cite{3264original} applies to determine the classes $[\cW]$ in $A^*(\mathcal{ X})$.
For our choice of $\cM$ the formula identifies $[\cW]$ with the 
$n-m$-graded part of the quotient
\[
\frac{c(\bigoplus_{i = 1}^k \cO(-e_i) ^{\oplus n_i})}{c(\bigoplus_{j = 0}^m \cO(-\deg(f_j)))}
\]
of the total Chern classes 
of $\bigoplus_{i = 1}^k \cO(-e_i)^{ \oplus n_i}$ and 
$\bigoplus_{j = 0}^m \cO(-\deg(f_j))$ respectively.
Let $Y_i$ denote the $i-$th hyperplane class in $\mathcal{X}$.
Then  $c( \cO(e_i)) = 1 + Y_i$
and we obtain
\begin{lemma}
The class $[\cW]$ in $A^*(\mathcal{X}) $ is the
degree $n-m$ homogeneous part of the expression
\[
\Sigma :=
\frac{\prod_{i = 1}^k  ( 1-Y_i)^{n_i}  }{\prod_{j = 0}^m  1 - \sum_{i = 1}^k \deg(f_j)_i Y_i   . }
\]
\end{lemma}
By abuse of notation we view $\Sigma$ as a power series in the formal variables $Y_1, \dots, Y_k$.
$\Sigma$ is a sum of homogeneous polynomials in $Y_1, \dots, Y_k$:
\[
\Sigma = \sum_{i = 0}^\infty \Sigma_i, \ \Sigma_i \in \Z[ Y_1, \dots, Y_k]_{(i)}.
\]
The two Lemmas above allow us to compute the algebraic degree of optimization:
\begin{theorem}
\label{thm: algebraic_degree_of_optimization}
The number of critical points of $f_0$ on the variety $V = V(f_1, \dots, f_m)$ is the coefficient of the monomial
$Y_1^{n_1}  \cdots Y_k^{n_k}$ in the product
\[
\Sigma_{n-m}
\prod_{j = 1}^m \left(1 -  \sum_{i = 1}^k \deg(f_j)_i Y_i   \right).
\]
\end{theorem}

\begin{example}
\label{ex: cubic_and_conics}
Consider the problem of optimizing a quadratic function $f_0$ over two quartic constraints $f_1, f_2$ in 4 variables $x_1, x_2, y_1, y_2$.
For a generic choice of such functions the number of critical points is $544$ \cite[Theorem 2.2]{algebraic_degree}. Compare this to the 
the case where $f_0$ is bilinear and both equality constraints are biquadratic. 
To apply theorem \ref{thm: algebraic_degree_of_optimization} let $X$ and $Y$ denote two formal variables.
First we expand the power series
\[
\frac{   (1-X)^2(1-Y)^2  }{    (1-(X+Y))(1-2(X+Y))^2        }
\]
and compute the degree $2$ homogeneous part $8X^2  + 18XY + 8Y^2$.
The algebraic degree of optimization is the coefficient $208$ of the monomial $X^2Y^2$
in 
$\left(2X + 2Y \right)^2 
\left( 8X  + 18XY + 8Y \right).
$
As we will verify below,  the bound $208$ is sharp.

\section{Numerical nonlinear algebra}
\label{sec: numerical_nonlinear_algebra}

In the following we demonstrate how bounds on the algebraic degree of optimization can be used in conjunction with
numerical methods, in order to compute optimizers for $\eqref{eq:rewardMaximization}$ together with a certificate for the correctness of the result.
The basic idea is to solve the Lagrangian system \eqref{eq:Lagrangian} using numerical methods such as the software package  \texttt{HomotopyContinuation.jl}. Although in principle 
it is possible to find all solutions, since this method involves numerical computations we might fail to find all critical points.
There is however a certification method implemented. This feature gives \emph{lower bounds} on the number of obtained critical points.

\begin{example}
Coming back to example \ref{ex: cubic_and_conics} we consider an optimization problem with objective function of bidegree $(1, 1)$
and equality constraints of bidegree $(2, 2)$ in the varible groups $x_1, x_2$ and $y_1, y_2$:
\small
\begin{polynomial}
f_0 = 4y_2x_1+17y_2x_2+13x_1y_1+48x_2y_1-15y_2+6x_1+23x_2+38y_1+24\\
f_1 =  x_1^2y_1^2 +x_1^2y_2^2 + x_2^2y_1^2 + x_2^2y_2^2 + 2x_1^2 + 3x_2^2 + 5y_1^2 + 7y_2^2 - 100, \\
f_2 = 
 20y_2^2x_1^2-18y_2^2x_1x_2-5y_2^2x_2^2+23y_2x_1^2y_1+28y_2x_1x_2y_ +49y_2x_2^2y_1-8x_1^2y_1^2+46x_1x_2y_1^2
 +18x_2^2y_1^2+30y_2^2x_1+26y_2^2x_2-19y_2x_1^2 +19y_2x_1x_2-9y_2x_1y_1+38y_2x_2^2+25y_2x_2y_1-17x_1^2y_1
 +28x_1x_2y_1-2x_1y_1^2-20x_2^2y_1+23x_2y_1^2+35y_2^2+46y_2x_1+29y_2x_2+4y_2y_1+6x_1^2+42x_1x_2+32x_1y_1
 -12x_2^2-13x_2y_1+39y_1^2+31y_2+50x_1-7x_2+22y_1-3 .
\end{polynomial}
The optimization problem has exactly $208$ complex critical points, $22$ of which are real.

The following Julia code solves the Lagrangian formulation \eqref{eq:Lagrangian} of the optimization problem using the software package \texttt{HomotopyContinuation.jl}.
It finds 208 critical points, certifies that all of them are distinct and that exactly 22 of them are real. This shows that the upper bound given by Theorem \ref{thm: algebraic_degree_of_optimization}
is attained.
In particular, by comparing the values of $f_0$ on the real ones, the optimizer can be determined.
\begin{small}
\begin{verbatim}
@var x_1 x_2 y_1 y_2 l1 l2;
variables = [x_1, x_2, y_1, y_2, l1, l2];
f_0, f_1, f_2 = ...;
L = f0 - l1 * f1 - l2 * f2;
∇L = differentiate(L, variables);
result = solve(System(∇L));
certificate = certify(∇L, result);
reals = real_solutions(result);
obj_values =
map(s -> evaluate(f0, [x_1, x_2, y_1, y_2]
=> s[1:4]), reals);
minval, minindex = findmin(obj_values);
minarg = reals[minindex][1:4];
\end{verbatim}
\end{small}
The objective function attains its minimal value $-234.876$ at the point $ \left( 0.011, 0.019, -6.355, 0.774 \right)$.
\end{example}

\paragraph{Outlook}
Theorem \ref{thm: algebraic_degree_of_optimization} can be stated analogously when taking the closure of $V$ and $W$ not in a product of projective spaces, but any toric variety.
It is desirable to obtain a version of the theorem that holds for polynomials that are generic relative to their Newton polytopes or even generic relative to their supports.
In particular, it is a natural question to ask when the BKK bound of the Lagrangian formulation \eqref{eq:Lagrangian} coincides with the algebraic degree of polynomial optimization.
This is the subject of our next paper.

From a computational perspective it would be interesting to make
the main theorem \ref{thm: algebraic_degree_of_optimization} effective by finding, for prescribed multidegrees, an instance of the problem \eqref{eq:rewardMaximization}
for which all critical points can be easily computed. This would provide an effective way to solve $\eqref{eq:rewardMaximization}$
by the means of homotopy continuation.
\end{example}

 {\small
\bibliographystyle{alpha}
\bibliography{bibliography}
}


\end{document}